\documentclass[11pt]{amsart}
\usepackage{amsmath,amssymb,amsthm,upref,graphicx,mathrsfs}
\usepackage[top=35mm, bottom=35mm, left=35mm, right=35mm]{geometry}
\usepackage{color,xcolor}
\usepackage[
  colorlinks=true,
  linkcolor=blue,
  citecolor=blue,
  urlcolor=blue]{hyperref}

\numberwithin{equation}{section}
\newtheorem{theorem}{Theorem}[section]
\newtheorem{proposition}[theorem]{Proposition}
\newtheorem{corollary}[theorem]{Corollary}
\newtheorem{lemma}[theorem]{Lemma}

\newtheorem{claim}[theorem]{Claim}

\numberwithin{equation}{section}

\begin{document}
\title[Ergodic behaviours of composition operators]{Ergodic behaviors of composition operators acting on space of bounded holomorphic functions}
\author{Hamzeh Keshavarzi, Karim Hedayatian}
\maketitle
\begin{abstract}
We completely characterize the mean ergodic composition operators on $H^\infty(\mathbb{B}_n)$. In particular,  we show that a composition operator acting on this space is mean ergodic if and only if it is uniformly mean ergodic.
\end{abstract}

\section{Introduction and main results}

The purpose of this paper is to prove the following theorem:

 \begin{theorem}
Let $\varphi$ be a holomorphic self-map of $\mathbb{B}_n$. Then, the following statements are equivalent.
\begin{itemize}
\item[(i)]  $C_\varphi$ is mean ergodic on $H^\infty(\mathbb{B}_n)$.
\item[(ii)]  $C_\varphi$ is uniformly mean ergodic on $H^\infty(\mathbb{B}_n)$.
\item[(iii)] $\varphi$ has a fixed point in $\mathbb{B}_n$ and there is a $k\in \mathbb{N}$ such that $\| \varphi_{kj}-\rho_\varphi \|_\infty \rightarrow 0$ as $j\rightarrow \infty$.
\end{itemize}
\end{theorem}
Where $\rho_\varphi$ is the holomorphic retraction associated with $\varphi$ and is defined below.  We prove this theorem in two parts: Theorems \ref{mt2} and \ref{mt3}. Moreover, Theorem \ref{mt1} plays a key role in our method. However, we believe that Theorem \ref{mt1} has an independent interest.

Throughout the paper, $n$ is a fixed positive integer. Here is some notations:
\begin{itemize}
\item $\mathbb{C}$: the complex plane.
\item $\mathbb{B}_n=\{z\in \mathbb{C}^n: \ |z|<1 \}$: the unit ball of $\mathbb{C}^n$.
\item $\mathbb{D}=\mathbb{B}_1$: the unit disk in $\mathbb{C}$.
\item $H(\mathbb{B}_n)$: the space of all holomorphic functions from $\mathbb{B}_n$ into $\mathbb{C}$
 \item $H^\infty(\mathbb{B}_n)$: the subspace of all bounded functions in $H(\mathbb{B}_n)$.
\item $Hol(\mathbb{B}_n, \mathbb{B}_n)$: the set of all holomorphic self-maps of $\mathbb{B}_n$
\end{itemize}

Consider $\varphi\in Hol(\mathbb{B}_n,\mathbb{B}_n)$. The iterates of $\varphi$ are the functions $\varphi_k:=\varphi\circ \stackrel{(k)}{...} \circ\varphi$. We denote by $\varphi^i$, $1\leq i\leq n$ the components of $\varphi$, that is, $\varphi=(\varphi^1,...\varphi^n)$ where $\varphi^i:\mathbb{B}_n\rightarrow \mathbb{C}$ are holomorphic functions. Moreover, the composition operator $C_\varphi$ on $H(\mathbb{B}_n)$ is defined as $C_\varphi f=f\circ \varphi$. 

When we say that $\rho\in Hol(\mathbb{B}_n,\mathbb{B}_n)$  is holomorphic retraction, it means that it is an idempotent, that is, $\rho_2 = \rho$. Clearly, if $\varphi :\mathbb{B}_n \rightarrow \mathbb{B}_n$ be holomorphic such that the sequence of its iterates converges to a holomorphic
function $h:\mathbb{B}_n \rightarrow \mathbb{B}_n$. Then, $h_2 = h$, that is, $h$ is a holomorphic retraction of $\mathbb{B}_n$. For more details about the holomorphic self-maps of the unit ball and their iterates see \cite[Chapter 2]{abate}.

Let $\varphi :\mathbb{B}_n \rightarrow \mathbb{B}_n$ be holomorphic and have an interior fixed point. Then, from \cite[Theorem 2.1.29 and Proposition 2.2.30]{abate}, there exist a unique submanifold $M_\varphi$
of $\mathbb{B}_n$ and a unique holomorphic retraction $\rho_\varphi:\mathbb{B}_n \rightarrow M_\varphi$ such that every limit point $h \in Hol(\mathbb{B}_n,\mathbb{B}_n)$
of $\{\varphi_j\}$ is of the form $h = \gamma \circ \rho_\varphi$,
where $\gamma$ is an automorphism of $M_\varphi$. Moreover, even $\rho_\varphi$ is a limit point of the sequence $\{\varphi_j\}$. This implies that $\rho_\varphi\circ \varphi=\varphi\circ \rho_\varphi$.

Let $\{e_1,...,e_n\}$ be the standard basis of $\mathbb{C}^n$.

\begin{theorem} \label{mt1}
Let $\varphi$ be a holomorphic self-map of the unit ball with converging iterates and $\varphi(0)=0$. Then, there is an invertible matrix $V$ so that:
$$V^{-1}\varphi_j V=\Big((V^{-1}\varphi_j V)^1,...,(V^{-1}\varphi_j V)^s\Big)\oplus P_{n-s},$$
where $\dim M_\varphi=n-s$, the functions $(V^{-1}\varphi_j V)^1$,...,$(V^{-1}\varphi_j V)^s$ are the components of $V^{-1}\varphi_j V$, and $P_{n-s}$ is the orthogonal projection from $\mathbb{C}^n$ onto $U=e_{s+1}\oplus ... \oplus e_n$. Moreover, $V^{-1}\varphi_j V$ coverges to $P_{n-s}$ uniformly on the compact subsets of $V^{-1}\mathbb{B}_n$.
\end{theorem}

Let $X$ be a Banach space and $T:X\rightarrow X$ be an operator. Then, We say that $T$ is mean ergodic if
$$M_j(T)=\dfrac{1}{j} \sum_{i=1}^j T^i.$$
converges to a bounded operator defined on $X$ for the strong operator topology. Uniformly mean ergodicity will define in a same way with convergence in the operator norm.

Lotz \cite{Lotz} proved that: If $X$ is a Grothendieck Banach space with Dunford-Pettis property (GDP space), and $T\in L(X)$ satisfies $\|T^n/n\|\rightarrow 0$, then $T$ is mean ergodic if and only if it is uniformly mean ergodic. For the definition of GDP spaces see \cite[Pages 208-209]{Lotz}.

For some work on the mean ergodicity of composition operators see \cite{arendt1,arendt2,beltran, bonet1,jorda, jornet, keshavarzi3}.
The (uniformly) mean ergodicity of composition operators on $H^\infty(\mathbb{D})$ have been characterized in \cite{beltran}. It is well-known that $H^\infty(\mathbb{D})$ is a GDP space. Thus, a composition operator, acting on $H^\infty(\mathbb{D})$, is mean ergodic if and only if it is uniformly mean ergodic. However, we do not know whether $H^\infty(\mathbb{B}_n)$ is a GDP space or not. In \cite{keshavarzi3}, the first author has proved that if $\varphi$ is a holomorphic self-map of the unit ball with converging iterates and an interior fixed point, then the mean ergodicity and the uniformly mean ergodicity of $C_\varphi$ are equivalent. In the following theorem, we give this equivalence for all $\varphi\in Hol(\mathbb{B}_n,\mathbb{B}_n)$ with an interior fixed point.

 \begin{theorem} \label{mt2}
Let $\varphi$ be a holomorphic self-map of the unit ball with a fixed point in $\mathbb{B}_n$. Then, the following statements are equivalent.
\begin{itemize}
\item[(i)]  $C_\varphi$ is mean ergodic on $H^\infty(\mathbb{B}_n)$.
\item[(ii)]  $C_\varphi$ is uniformly mean ergodic on $H^\infty(\mathbb{B}_n)$.
\item[(iii)] There is a $k\in \mathbb{N}$ such that $\| \varphi_{kj}-\rho_\varphi \|_\infty \rightarrow 0$, as $j\rightarrow \infty$.
\end{itemize}
\end{theorem}

As the final result, we prove that every holomorphic self-map of $\mathbb{B}_n$ which has no interior fixed point induces a composition that is not mean ergodic on $H^\infty(\mathbb{B}_n)$. This theorem gives the answer to \cite[Question 3.16]{keshavarzi3}.

\begin{theorem} \label{mt3}
Let the holomorphic function $\varphi:\mathbb{B}_n\rightarrow \mathbb{B}_n$ has no interior fixed point. Then, $C_\varphi$ is not mean ergodic on $H^\infty(\mathbb{B}_n)$.
\end{theorem}

\section{Basic results}
Every automorphism $\varphi$ of $\mathbb{B}_n$ is of the form $\varphi = U\varphi_a= \varphi_b V$,
where $U$ and $V$ are unitary matrices of $\mathbb{C}^n$ and
\begin{equation}\label{e8}
\varphi_a(z)=\dfrac{a-P_a(z)-s_aQ_a(z)}{1-\langle z,a\rangle}, \qquad z\in \mathbb{B}_n,
\end{equation}
where $a\neq 0$, $s_a=\sqrt{1-|a|^2}$, $P_a$ is the projection from $\mathbb{C}^n$ onto the subspace $\langle a \rangle$ spanned by $a$, and $Q_a$ is the projection from $\mathbb{C}^n$ onto  $\mathbb{C}^n\ominus \langle a \rangle$. Clearly, $\varphi_a(0)=a$, $\varphi_a(a)=0$, and $\varphi_a \circ \varphi_a(z)=z$.
It is well-known that an automorphism $\varphi$ of $\mathbb{B}_n$ is a unitary matrix of $\mathbb{C}^n$ if and only if $\varphi(0)=0$. 

Let $\Omega$ be a strongly pseudoconvex bounded domain.
The infinitesimal Kobayashi metric $F_K:\Omega\times \mathbb{C}^n \rightarrow [0,\infty)$ is defined as:
$$F_K(z,w)=\inf  \Big{\{} C>0: \ \exists f\in H(\mathbb{D}, \Omega) \ with \ f(0)=z, \ f^\prime(0)=\dfrac{w}{C} \Big{\}},$$
where $H(\mathbb{D}, \Omega)$ is the space of analytic functions from $\mathbb{D}$ to $\Omega$.
Let $\gamma:[0,1]\rightarrow \Omega$ be a $C^1$-curve. The Kobayashi length of $\gamma$ is defined as:
$$L_K(\gamma)= \int_0^1 F_K (\gamma(t), \gamma^\prime(t))dt.$$
For $z,w\in \Omega$, the Kobayashi metric function is defined as:
$$k_\Omega(z,w)=\inf \Big{ \{} L_K(\gamma); \ \gamma \ is \ C ^1-curve \ with \ \gamma(0)=z \  and \ \gamma(1)=w \Big{\}}.$$

If $\Omega$ and $\Lambda$ are two strongly pseudoconvex bounded domains and  $\varphi:\Omega\rightarrow\Lambda$ is a holomorphic function, then from \cite[Proposition 2.3.1]{abate}, we have:
\begin{equation}\label{e5}
k_\Lambda(\varphi(z),\varphi(w))\leq k_\Omega(z,w), \qquad \forall z,w\in\Omega.
\end{equation}
Thus, $k_\Omega$ is invariant under automorphisms, that is,
$$k_\Omega(\varphi(z),\varphi(w))=k_\Omega(z,w),$$
for all $z,w\in \mathbb{B}_n$ and $\varphi:\Omega\rightarrow \Omega$ is an automorphism.

Let $\beta$ from $\mathbb{B}_n \times \mathbb{B}_n $ to $ [0,\infty)$ be the Bergman metric.  From \cite[Corollary 2.3.6]{abate}, the Kobayashi metric and the Bergman metric coincide on $\mathbb{B}_n$. We have:
\begin{equation} \label{e0}
\beta(z,w)= \dfrac{1}{2} \log \dfrac{1+|\varphi_z(w)|}{1-|\varphi_z(w)|}, \qquad z,w\in \mathbb{B}_n.
\end{equation}
We shall denote by $B(a,r)$ the Bergman ball centered at $a\in \mathbb{B}_n$ with radius $r>0$, that is,
$$B(a,r)=\{z\in \mathbb{B}_n: \ \beta(a,z)<r\}.$$
It is well-known (see \cite[page 134]{abate}) that $B(a,r)$ is the ellipsoid
\begin{equation} \label{e11}
\dfrac{|P_a(\zeta)-a_r|^2}{R^2 s^2}+\dfrac{|Q_a(\zeta)|^2}{R^2s}<1,
\end{equation}
where $R=\tanh r$, $a_r=\frac{1-R^2}{1-R^2|a|^2}a$ and $s=\frac{1-|a|^2}{1-R^2|a|^2}$.

Let $P_k$ be the space homogeneous polynomial $P:\mathbb{B}_n\rightarrow \mathbb{C}$ of degree $k$. The Taylor series expansions of functions in $H^\infty(\mathbb{B}_n)$
yield a direct sum decomposition of
$$H^\infty(\mathbb{B}_n) = P_0\oplus P_1\oplus ... \oplus P_m \oplus R_m;$$
where the remaining space $R_m$ consists of the functions $h\in H^\infty(\mathbb{B}_n)$ such
that $|h(z)|/\|z\|^m$ is bounded for $z$ near $0$. Similarly, $f:\mathbb{B}_n\rightarrow \mathbb{C}^n$ admits a homogeneous expansion:
$$f(z)=\sum_{k=0}^\infty F_k(z)=f(0)+f^\prime(0)z+...,$$
where all $n$ component functions of each $F_k$ are homogeneous polynomial of degree $k$.

It should be noted that $d_z \varphi=\varphi^\prime(z)$. Note that $d_z \varphi$ is a matrix:
 $$d_z \varphi:= \begin{bmatrix}
\frac{\partial \varphi^1}{\partial z_1} & \cdots & \frac{\partial \varphi^1}{\partial z_n}\\
\cdots & \cdots & \cdots\\
\frac{\partial \varphi^n}{\partial z_1} & \cdots & \frac{\partial \varphi^n}{\partial z_n}
\end{bmatrix} (z).$$

\section{Proof of Theorem \ref{mt1}}

Let $n-s$ be the dimension of $M_\varphi$.

If $s=0$, then from \cite[Proposition 2.2.14]{abate} and \cite[Proposition 3.8]{keshavarzi3}, $\varphi$ is a unitary matrix. Since the iterates of $\varphi$ are convergent, $\varphi$ is the identity matrix.
If $s=n$, then from \cite[Theorem 2.2.32]{abate}, $M_\varphi=\{0\}$ and $\rho_\varphi\equiv 0$. Therefore, for $s=0$ or $n$, the result is obtained by considering $V$ as the identity matrix.

Thus, let $1\leq s\leq n-1$.
We give the proof in three steps:

\subsection*{Step 1} There is an invertible matrix $V$ so that $V^{-1} d_0 \rho V=P_{n-s}$.
\begin{proof}
Recall that $P_{n-s}$ is the orthogonal projection from $\mathbb{C}^n$ onto $e_{s+1}\oplus...\oplus e_n$.

 Let $V$ be an invertible matrix so that $V^{-1} d_0 \rho V$ be the Jordan canonical form of $d_0 \rho$. Since, $\rho^2=\rho$ and $\rho(0)=0$, the matrix $d_0\rho$ is also an idempotent. Thus, the eigenvalues of $V^{-1} d_0 \rho V$ are in $\{0,1\}$. Note that since $\rho(\mathbb{B}_n)=M$ and  $\rho$ is identity on $M$, it is easy to show that $0$ and $1$ will be repeated $s$ and $n-s$ times as the eigenvalues of $d_0\rho$, respectively.

 We have
\begin{equation*}
V^{-1} d_0 \rho V=J_1(0)\oplus...\oplus J_k(0)\oplus I_1(1)\oplus ...\oplus I_l(1),
\end{equation*}
where $J_i(0)$ and $I_i(1)$ are the blocks associated with the eigenvalues $0$ and $1$, respectively. Now since $d_0\rho$ is an idempotent, the blocks $J_i(0)$ and $I_i(1)$ must be $1\times 1$.
That is,
\begin{equation*}
V^{-1} d_0\rho V=
\begin{bmatrix}
0 & 0   \\
0  & I_{n-s} \\
\end{bmatrix},
\end{equation*}
where $I_{n-s}$ is the $(n-s)\times(n-s)$ identity matrix. Hence,
 $V^{-1} d_0 \rho V=P_{n-s}$.
\end{proof}

From \cite[Theorem 2.1.21]{abate}, we know that if $f:\mathbb{B}_n\rightarrow \mathbb{B}_n$ is holomorphic, $f(0)=0$ and $d_0 f$ is identity, then so is $f$. In the next step, we want to show that if $d_0 f=0_s\oplus  I_{n-s}$, then $f=0_s\oplus  I_{n-s}$.

\subsection*{Step 2} For the matrix $V$, obtained in step 1, we have $V^{-1} \rho V=P_{n-s}$.
\begin{proof}
Let $V^{-1} \rho V\neq P_{n-s}$.
 Consider the function $\psi=V^{-1} \rho V-P_{n-s}:\mathbb{B}_n\rightarrow \mathbb{C}^n$. Since $d_0 \psi=V^{-1} d_0\rho V-d_0P_{n-s}=0$, $\psi(0)=0$, but $\psi\neq 0$, we can write:
$$V^{-1} \rho V(z)=P_{n-s} (z)+ F_k(z) +\sum_{j=k+1}^\infty F_j(z),$$
where $F_k$ is a homogeneous polynomial of degree $k\geq 2$. In summation, $F_j$ is zero or a homogeneous polynomial of degree $j$.

Note that every component of a homogeneous polynomial of degree $j$ is a summation of polynomials
$$z^m=z_1^{m_1}...z_n^{m_n},$$
where $z=(z_1,...,z_n)$, $m=(m_1,...,m_n)\in \mathbb{N}^n$, and $m_1+...+m_n=j$.
Thus, for $j\geq k$ if $F_j=(F_j^1,...,F_j^n)$ is non-zero, then each component of $F_j(V^{-1} \rho V(z))$ is a summation of polynomials
\begin{align*}
&\Big(P_{n-s} (z)+ F_k(z) +\sum_{j=k+1}^\infty F_j(z)\Big)^m\\
&\qquad \ \ \ =\Big(F_k^1(z) +\sum_{j=k+1}^\infty F_j^1(z)\Big)^{m_1}...\Big(F_k^s(z) +\sum_{j=k+1}^\infty F_j^s(z)\Big)^{m_s}\\
 &\qquad \qquad \times \Big(z_{s+1}+F_k^{s+1}(z) +\sum_{j=k+1}^\infty F_j^{s+1}(z)\Big)^{m_{s+1}}...\Big(z_n+F_k^n(z) +\sum_{j=k+1}^\infty F_j^n(z)\Big)^{m_n}
\end{align*}
Thus, from the above statement and the assumption $1\leq s\leq n-1$, if $F_j$ is non-zero for $j\geq k$, then each component of $F_j(V^{-1} \rho V(z))$ is a polynomial with a degree greater than or equal to:
$$km_1+...+km_s+m_{s+1}+...+m_n.$$
 On the other hand, since $k\geq 2$, we have
$$km_1+...+km_s+m_{s+1}+...+m_n>\sum_{i=1}^n m_i=j.$$
  Thus,
$$V^{-1}\rho^2 V(z)= P_{n-s} (z)+P_{n-s} F_k(z) +\sum_{j=k+1}^\infty G_j(z),$$
where each $G_j$ is zero or a homogeneous polynomial of degree $j$.
 Since $\rho^2=\rho$, we must have $F_k=P_{n-s} F_k$ which contradicts the assumption that $s\neq 0, n$.
\end{proof}

Indeed, we proved the following result in steps $1$ and $2$ as well as the paragraph before them:
\begin{corollary}
Every holomorphic retraction $\rho$ on $\mathbb{B}_n$ which fixes the origin is a matrix.
\end{corollary}

\subsection*{Step 3} $(V^{-1}\varphi_j V)^{i}(z^1,...,z^n)=z^i$, for $i=s+1,...,n$ and $j\in \mathbb{N}$.
\begin{proof}
From step 2,
\begin{equation} \label{e6}
V^{-1}(\rho\circ \varphi) V=V^{-1}\circ\rho\circ V( V^{-1}\circ \varphi \circ V)=0_s\oplus \begin{bmatrix}
(V^{-1}\varphi V)^{s+1}(z)   \\
\vdots \\
(V^{-1}\varphi V)^n(z)\\
\end{bmatrix}.
\end{equation}
Moreover, since $\rho$ and $\varphi\circ \rho$  are the limit points of the convergent sequence $\{\varphi_j\}$, we have:
\begin{equation}\label{e7}
V^{-1}\rho\circ \varphi V=V^{-1}\varphi\circ \rho V=V^{-1} \rho V.
\end{equation}
Thus, \ref{e6}, \ref{e7}, and step 2 imply that
\begin{equation*}
\begin{bmatrix}
(V^{-1}\varphi V)^{s+1}(z)   \\
\vdots \\
(V^{-1}\varphi V)^n(z)\\
\end{bmatrix}= \begin{bmatrix}
z^{s+1}   \\
\vdots \\
z^n\\
\end{bmatrix}.
\end{equation*}
Again, by a similar argument, we can see that $\rho\circ \varphi_j=\varphi_j\circ \rho=\rho$. Thus,
\begin{equation*}
\begin{bmatrix}
(V^{-1}\varphi_j V)^{s+1}(z)   \\
\vdots \\
(V^{-1}\varphi_j V)^n(z)\\
\end{bmatrix}
=\begin{bmatrix}
z^{s+1}   \\
\vdots \\
z^n\\
\end{bmatrix}.
\end{equation*}
The proof is complete.
\end{proof}

\section{Proof of Theorem \ref{mt2}}

If $\varphi$ has an interior fixed point $a\in \mathbb{B}$, then $\psi:=\varphi_a \circ \varphi\circ \varphi_a$ is a holomorphic self-map of $\mathbb{B}_n$ that $\psi(0)=0$. Hence, without loss of generality, we assume that $\varphi(0)=0$. (ii)$\Rightarrow$ (i) is obvious.

\subsection{(iii)$\Rightarrow$ (ii)} 
 Since $\varphi_{kj}\rightarrow\rho$, from Theorem \eqref{mt1}, there is an invertible matrix $V$ so that
$$V^{-1}\varphi_{kj} V=\Big((V^{-1}\varphi_{kj} V)^1,...,(V^{-1}\varphi_{kj} V)^s\Big)\oplus P_{n-s},$$
and
\begin{equation*}
V^{-1} \rho V=
\begin{bmatrix}
0 & 0   \\
0  & I_{n-s} \\
\end{bmatrix}=P_{n-s}.
\end{equation*}
From the continuity of $V^{-1}$ and (iii), there is a $C>0$ so that
\begin{align} \label{e12}
&\lim_{j\rightarrow \infty}\sup_{z\in V^{-1}\mathbb{B}_n} \Big{|}((V^{-1}\varphi_{kj} V)^1,...,(V^{-1}\varphi_{kj} V)^s)(z)\Big{|}\notag\\
&\qquad \qquad \qquad \qquad \qquad \qquad=\lim_{j\rightarrow \infty}\|V^{-1} (\varphi_{kj}- \rho) \|_\infty\notag\\
&\qquad \qquad  \qquad \qquad \qquad \qquad \leq C \lim_{j\rightarrow \infty}\| \varphi_{kj}- \rho \|_\infty =0.
\end{align}
It is easy to see that $V^{-1}\mathbb{B}_n$ is a taut manifold. Thus, from \ref{e5} we have:
\begin{align*}
\sup_{z\in\mathbb{B}_n} \beta(\varphi_{kj} (z),\rho(z))&\leq \sup_{z\in \mathbb{B}_n} k_{V^{-1}\mathbb{B}_n}(V^{-1}\varphi_{kj}(z),V^{-1}\rho (z))\\
&= \sup_{z\in V^{-1}\mathbb{B}_n} k_{V^{-1}\mathbb{B}_n}(V^{-1}\varphi_{kj} V(z),V^{-1}\rho V(z)).
\end{align*}
Hence, from  \cite[Lemma 4.1]{keshavarzi3} and Equation \ref{e12}, we obtain:
\begin{align*}
\sup_{z\in\mathbb{B}_n} \beta(\varphi_{kj} (z),\rho(z))&\leq\sup_{z\in V^{-1}\mathbb{B}_n} \omega(\Big{|}((V^{-1}\varphi_{kj} V)^1,...,(V^{-1}\varphi_{kj} V)^s)(z)\Big{|},0)\\
&= \dfrac{1}{2} \sup_{z\in V^{-1}\mathbb{B}_n} \tanh^{-1} (\Big{|}((V^{-1}\varphi_{kj} V)^1,...,(V^{-1}\varphi_{kj} V)^s)(z)\Big{|})\rightarrow 0,
\end{align*}
as $j\rightarrow \infty$. Therefore, (ii) follows from  \cite[Theorem 3.6]{keshavarzi3}.

\subsection{(i)$\Rightarrow$ (iii)} Before presenting the proof, we state some auxiliary results.

For $k>0$ and $\zeta\in \partial \mathbb{B}_n$, we define the ellipsoid
$$E(k,\zeta)=\{ z\in \mathbb{B}_n: \ |1-\langle z,\zeta\rangle|^2\leq k(1-|z|^2)\}.$$
Let $\rho$ be a holomorphic self-map of the unit ball and $\eta>0$. Set
$$L(\rho,\eta)=\{z\in \mathbb{B}_n, \ \ \beta (z,\rho(z))\geq\eta\}.$$
The following lemma is an extension of \cite[Lemma 3.9]{keshavarzi3}. Since the proof is the same, we omit it.

\begin{lemma} \label{l6}
Let $\varphi$ be a holomorphic self-map of the unit ball, $\varphi(0)=0$, and $\rho$ be the holomorphic retraction associated with $\varphi$.  If $\eta>0$ be such that $L(\rho,\eta)\neq \emptyset$, then there is some $A>1$ such that
$$\dfrac{1-|\varphi(z)|}{1-|z|}>A, \qquad \forall z\in L(\rho,\eta).$$
\end{lemma}

\begin{proposition} \label{p4}
$\beta(z,w)\geq \dfrac{1}{2}|z-w|$, for all $z,w\in \mathbb{B}_n$.
\end{proposition}
\begin{proof}
The case $z=w$ is clear. Let $z\neq w$. Then $\beta(z,w)=r>0$. Note from \eqref{e11} that $B(w,r)$ is the ellipsoid
$$\dfrac{|P_w(\zeta)-w_R|^2}{R^2 s^2}+\dfrac{|Q_w(\zeta)|^2}{R^2s}<1,$$
where
$$R=\tanh r=\dfrac{e^r-e^{-r}}{e^r+e^{-r}}<1,$$
 $w_R=\frac{1-R^2}{1-R^2|w|^2}w$ and $s=\frac{1-|w|^2}{1-R^2|w|^2}<1$.
Thus,
$$\dfrac{|P_w(z)-w_R|^2}{R^2 s^2}+\dfrac{|Q_w(z)|^2}{R^2 s}=1,$$
 Since $s<1$ and $Q_w(z)$ is orthogonal to $P_w(z)$ and $P_w(z)-w_R$, we obtain
\begin{align*}
|z-w_R|^2&=|P_w(z)-w_R|^2+|Q_w(z)|^2\\
&=R^2 s \Big(\dfrac{|P_w(z)-w_R|^2}{R^2 s}+\dfrac{|Q_w(z)|^2}{R^2 s}\Big)\\
&< R^2 s \Big(\dfrac{|P_w(z)-w_R|^2}{R^2 s^2}+\dfrac{|Q_w(z)|^2}{R^2 s}\Big)=R^2 s.
\end{align*}
From the mean value theorem, there is a $0\leq t\leq r$ so that:
$$R=\tanh r=r sech^2 t \leq r.$$
Note that the last inequality comes from $sech \ t=\frac{2}{e^t+e^{-t}}\leq 1$.

Combining the above estimates, we deduce that:
\begin{align*}
|z-w|&\leq |z-w_R|+|w_R-w|\\
&< R\sqrt{s}+R^2\Big( \dfrac{1-|w|^2}{1-R^2 |w|^2}\Big)\\
&< 2R\leq 2r=2\beta(z,w).
\end{align*}
The proof is complete.
\end{proof}

Now we proceed to the proof of (i)$\Rightarrow$ (iii).
 From \cite[Lemma 3.3]{keshavarzi3}, there is a positive integer $k$ so that $\varphi_{kj}\rightarrow \rho$ uniformly on the compact subsets of $\mathbb{B}_n$ and
\begin{equation} \label{e4}
\lim_{j\rightarrow\infty} M_j(C_\varphi )= \frac{1}{k}\sum_{i=0}^{k-1} C_{\rho\circ \varphi_i}.
\end{equation}
 for the strong operator topology.
 Let (iii) not hold.

 \begin{claim} \label{cl0}
There is an $\varepsilon>0$ so that $\|\varphi_{kj}-\rho\|_\infty\geq \varepsilon$ for all $j$.
 \end{claim}
 \begin{proof}
Since (iii) does not hold, there is a sequence $m_j$ in $\mathbb{N}$ such that $\|\varphi_{km_j}-\rho\|_\infty\geq \varepsilon$ for all $j$. Consider an arbitrary positive integer $j$. Then, there is a $j_0$ so that $m_{j_0}\geq j$. Thus, from the fact that $\rho\circ \varphi_{kl}=\rho$ for all $l\in \mathbb{N}$, we have:
 \begin{align*}
 \varepsilon &\leq \|\varphi_{km_{j_0}}-\rho\|_\infty\\
 &=\|\varphi_{kj}\circ \varphi_{k(m_{j_0}-j)}-\rho\circ \varphi_{k(m_{j_0}-j)}\|_\infty\\
&=\sup_{z\in \mathbb{B}_n}|(\varphi_{kj}-\rho)( \varphi_{k(m_{j_0}-j)}(z))|\\
& \leq \|\varphi_{kj}-\rho\|_\infty.
 \end{align*}
The proof is complete.
 \end{proof}

\begin{claim} \label{cl1}
 For every $0<r<1$, we can find $a\in \mathbb{B}_n$ and $m\in \mathbb{N}$ such that:
$$|\varphi_{2km}(a)-\rho(a)|\geq \varepsilon, \ \ and \ \ |\varphi_{km}(a)|>r.$$
\end{claim}
\begin{proof}
If the claim do not hold, then there is an $0<r<1$ so that
\begin{equation} \label{e13}
\sup\{|\varphi_{2kj}(z)-\rho(z)|; \ z\in \mathbb{B}_n, \ |\varphi_{kj}(z)|>r\}\leq\varepsilon.
\end{equation}
for all $j\in \mathbb{N}$. On the other hand, there is a $j_0$ so that
\begin{equation} \label{e14}
\sup\{|\varphi_{kj_0}(z)-\rho(z)|; \ z\in \mathbb{B}_n, \ |z|\leq r\}\leq\varepsilon.
\end{equation}
We have:
\begin{align*}
\|\varphi_{2kj_0}-\rho\|_\infty&=\max \Big{\{} \sup\{|\varphi_{2kj}(z)-\rho(z)|; \ z\in \mathbb{B}_n, \ |\varphi_{kj}(z)|>r\},\\
&\qquad \qquad \sup\{|\varphi_{2kj}(z)-\rho(z)|; \ z\in \mathbb{B}_n, \ |\varphi_{kj}(z)|\leq r\}\Big{\}}.
\end{align*}
From \eqref{e13}, the first supremum is less than or equal to $\varepsilon$. For the second one, from the fact $\rho=\rho\circ \varphi_{kj_0}$ and \eqref{e14}, we have
\begin{align*}
&\sup\{|\varphi_{2kj_0}(z)-\rho(z)|; \ z\in \mathbb{B}_n, \ |\varphi_{kj_0}(z)|\leq r\}\\
&\qquad \qquad =\sup\{|\varphi_{kj_0}\circ \varphi_{kj_0}(z)-\rho\circ \varphi_{kj_0}(z)|; \ z\in \mathbb{B}_n, \ |\varphi_{kj}(z)|\leq r\}\\
&\qquad \qquad \leq\sup\{|\varphi_{kj_0}(z)-\rho(z)|; \ z\in \mathbb{B}_n, \ |z|\leq r\}\leq\varepsilon
\end{align*}
Therefore, $\|\varphi_{2kj_0}-\rho\|_\infty\leq \varepsilon$, which contradicts Claim \eqref{cl0}.
 \end{proof}

 \begin{claim} \label{cl3}
 There are two sequences $\{m_j\}\subseteq \mathbb{N}$ and $\{a_j\}\subset\mathbb{B}_n$ and some $f$ in $H^\infty(\mathbb{B}_n)$ such that $|\varphi_{2km_j}(a_j)-\rho(a_j)|\geq\varepsilon$ for all $j$, and
$$f\circ \rho\equiv 0, \ \ f(\varphi_l(a_j))=|\varphi_{2km_j}(a_j)-\rho(a_j)|^2,\qquad  \ 1\leq l\leq km_j, \ \forall j\in \mathbb{N}.$$
\end{claim}
\begin{proof}
 From Lemma \ref{l6}, there is a  constant $0<a<1$ such that if $\beta(z,\rho(z))\geq \varepsilon/2$, then
   \begin{equation} \label{e3}
   \dfrac{1-|z|}{1-|\varphi(z)|}< a.
   \end{equation}

   Let $a_1$ in $\mathbb{B}_n$ be such that $|\varphi_{2k}(a_1)-\rho(a_1)|\geq \varepsilon$. Then, from Proposition \eqref{p4}, the fact that $\rho\circ \varphi_{kl}=\rho$ and $\rho \circ \varphi_l=\varphi_l\circ \rho$ for al $l\in \mathbb{N}$, and inequality \eqref{e5}, we obtain:
\begin{align*}
\dfrac{\varepsilon}{2}\leq\beta(\varphi_{2k}(a_1),\rho(a_1))=\beta(\varphi_{2k}(a_1),\rho\circ \varphi_{2k}(a_1))\leq \beta(\varphi_i(a_1),\rho\circ \varphi_i(a_1)).
\end{align*}
for all $1\leq i\leq 2k$. Thus, from \ref{e3}, we have
$$\dfrac{1-|\varphi_i(a_1)|}{1-|\varphi_{i+1}(a_1)|}< a, \ \ 1\leq i\leq k-1.$$

Put $m_1=1$. Using Claim \ref{cl1}, we can find $a_2\in \mathbb{B}_n$ and $m_2\in \mathbb{N}$ such that $|\varphi_{km_2}(a_2)|$ is large enough so that
$$|\varphi_{2km_2}(a_2)-\rho(a_2)|\geq \varepsilon,$$
and
$$\dfrac{1-|\varphi_{km_2}(a_2)|}{1-|\varphi(a_1)|}<a.$$
Again,
\begin{align*}
\dfrac{\varepsilon}{2}\leq\beta(\varphi_{2km_2}(a_2),\rho(a_2))\leq \beta(\varphi_i(a_2),\rho\circ \varphi_i(a_2))
\end{align*}
for all $0\leq i\leq 2km_2$. Thus, from \ref{e3}, we obtain:
$$\dfrac{1-|\varphi_i(a_2)|}{1-|\varphi_{i+1}(a_2)|}<a, \ \ 1\leq i\leq km_2-1.$$
By repeating  this process we will construct the sequence
\begin{equation*}
\begin{matrix}
x_1=\varphi_{k}(a_1),& x_2=\varphi_{k-1}(a_1), & ..., & x_{km_1}=\varphi(a_1)\\
x_{km_1+1}=\varphi_{km_2}(a_2), & x_{km_1+2}=\varphi_{km_2-1}(a_2),& ..., & x_{k(m_2+m_1)}=\varphi(a_2) \\
x_{k(m_2+m_1)+1}=\varphi_{km_3}(a_3), & x_{k(m_2+m_1)+2}=\varphi_{km_3-1}(a_3), & ...,& x_{k(m_3+m_2+m_1)}=\varphi(a_3)\\
\vdots & \vdots & \vdots & \ddots
\end{matrix},
\end{equation*}
which satisfies condition (i) of \cite[Lemma 3.11]{keshavarzi3}. Thus, there are some $M>0$ and a sequence $\{f_{l,j}\}_{j,l=1}^{\infty,km_j} \subset H^\infty(\mathbb{B}_n)$ such that
\begin{itemize}
\item[(a)] $f_{l,j}(\varphi_l(a_j))=1$, and $f_{l,j}(\varphi_r(a_s))=0$ whenever $l\neq r$ or $j\neq s$.
\item[(b)] $\sum_{j=1}^{\infty}\sum_{l=1}^{km_j} |f_{l,j}(z)|\leq M$, for all $z\in \mathbb{B}_n$.
\end{itemize}
Define
\begin{equation*}
f(z)=\sum_{j=1}^{\infty}\sum_{l=1}^{km_j}  \langle \varphi_{2km_j-l}(z)-\rho\circ \varphi_{2km_j-l}(z),\varphi_{2km_j}(a_j)-\rho(a_j)\rangle f_{l,j}(z).
\end{equation*}
Hence, from the Lebesgue dominated convergence theorem, (a), (b), and the fact that $\rho\circ \varphi=\varphi\circ \rho$, we deduce that  $f\in H^\infty(\mathbb{B}_n)$, $f(\rho)=0$, and
$$f(\varphi_l(a_j))=|\varphi_{2km_j}(a_j)-\rho(a_j)|^2, \ \ 1\leq j<\infty, \ 1\leq l\leq km_j.$$
 The proof is complete.
 \end{proof}

Using Claim \ref{cl3}, we have:
\begin{eqnarray*}
\Big{\|}\dfrac{1}{m_j} \sum_{l=1}^{m_j} C_{\varphi_l} - \frac{1}{k}\sum_{i=0}^{k-1} C_{\rho\circ \varphi_i} \Big{\|} &\geq&
\dfrac{1}{\|f\|_\infty} \Big{\|}\dfrac{1}{m_j} \sum_{l=1}^{m_j} C_{\varphi_l} f-\frac{1}{k}\sum_{i=0}^{k-1} C_{\rho\circ \varphi_i}  f \Big{\|}_\infty\\
&\geq&\dfrac{1}{\|f\|_\infty} \Big|\dfrac{1}{m_j} \sum_{l=1}^{m_j}  f(\varphi_l(a_j))-\frac{1}{k}\sum_{i=0}^{k-1} f(\rho\circ \varphi_i(a_j)\Big|\\
&=& \dfrac{1}{\|f\|_\infty} . \dfrac{1}{m_j} \sum_{l=1}^{m_j} |\varphi_{2km_j}(a_j)-\rho(a_j)|^2 \geq \dfrac{\varepsilon^2}{\|f\|_\infty} .
\end{eqnarray*}
From the above estimate, we deduce that $\{M_{j}(C_\varphi )\}_{j=1}^\infty$ does not converge to $\frac{1}{k}\sum_{i=0}^{k-1} C_{\rho\circ \varphi_i}$ for the strong operator topology, which contradicts \ref{e4}. Thus, (iii) holds.

\section{Proof of Theorem \ref{mt3}}
First, we define the sequence of operators $T_{j}:H^\infty(\mathbb{D})\rightarrow H^\infty(\mathbb{D})$ as follows
$$T_{j}f(z):= f \circ \varphi_j^1 (z,0,...,0),$$
where $\varphi_j^1$ is the first component of $\varphi_j$. Note that if we consider $f\in H^\infty(\mathbb{D})$ as a function in $ H^\infty(\mathbb{B}_n)$, then $T_jf=C_{\varphi_j}f$. Thus, if $C_\varphi$ is mean ergodic on $H^\infty(\mathbb{B}_n)$, then
$$N_j(\varphi):=\frac{1}{j}\sum_{i=1}^j T_{i}:H^\infty(\mathbb{D})\rightarrow H^\infty(\mathbb{D}),$$
converges for the strong operator topology.

We give the proof in two steps. In the first step, we show that if $N_j(\varphi)$ is SOT-convergent, then it must converge in the norm operator. Then, in the second step, we prove that $N_j(\varphi)$ does not converge in the norm operator. Therefore, the proof will be complete.

\subsection*{Step 1}
From the ergodic theorem, $M_j(\varphi)$ converges to a projection $P$ so that $PC_\varphi=C_\varphi P=P$.
Since $N_j(\varphi)$ converges for the strong operator topology to $P\mid_{H^\infty(\mathbb{D})}$ and $H^\infty(\mathbb{D})$ is a GDP space, from \cite[Theorem 2]{Lotz} the spectral radius of $N_j(\varphi)-P$ converges to $0$ as $j\rightarrow\infty$.
That is, $I-N_j(\varphi)+P$ is invertible for a large enough $j$.

 Now, we show that $I-T_1+P$ is bounded below. If not, then there is a sequence of unit vectors $\{f_l\}$ in $H^\infty(\mathbb{B}_n)$ so that:
$$\|(I-T_1+P)f_l\|_\infty\rightarrow 0 \qquad as \ j\rightarrow\infty.$$
Since $P=PT_1=P^2$, we obtain
$$\|Pf_l\|_\infty=\|P(I-T_1+P)f_l\|_\infty\rightarrow 0 \qquad as \ j\rightarrow\infty.$$
Thus,
$$\|(I-T_1)f_l\|_\infty\rightarrow 0 \qquad as \ j\rightarrow\infty.$$
Therefore,
\begin{align*}
(I-N_j(\varphi)+P) f_l&= (I-M_j(\varphi)+P) f_l\\
&= \dfrac{1}{n} \sum_{i=1}^j (I-C_{\varphi_i})f_l+Pf_l\\
&= \dfrac{1}{n} \sum_{i=1}^j (I+C_\varphi+...+C_{\varphi_{i-1}})(I-C_{\varphi})f_l+Pf_l\\
&= \dfrac{1}{n} \sum_{i=1}^j (I+T_1+...+T_{i-1})(I-T_1)f_l+Pf_l\rightarrow 0,\\
\end{align*}
as $l\rightarrow\infty$. This contradicts the invertibility of  $I-N_j(\varphi)+P$.

Now, since $I-T_1+P$ is bounded below, there is a bounded operator $S$ on $H^\infty(\mathbb{D})$ so that $S(I-T_1+P)=I$. Therefore,
\begin{align*}
(N_j(\varphi)-P)&=S(I-T_1+P) (N_j(\varphi)-P)\\
&=S(I-C_\varphi+P) (M_j(\varphi)-P)\\
&=\dfrac{1}{j} S(C_\varphi-C_{\varphi_{j+1}})\rightarrow 0,
\end{align*}
as $j\rightarrow \infty$.
\subsection*{Step 2}
The proof of this step is similar to that of \cite[Theorem 3.6]{beltran} and also \cite[Theorem 3.14]{keshavarzi3}.

From \cite[Theorem 2.2.31]{abate}, there is a $z_0 \in \partial \mathbb{B}_n$ such that $\varphi_j\rightarrow z_0$ uniformly on the compact subsets of $\mathbb{B}_n$. By a unitary equivalent, we can let $z_0=e_1$. Thus, if $\varphi_j=(\varphi_j^1,...,\varphi_j^n)$, then $\varphi_j^1\rightarrow 1$ and $\varphi_j^i\rightarrow 0$ for $2\leq i\leq n$ uniformly on the compact subsets of $\mathbb{B}_n$ as $j\rightarrow\infty$.

Thus, if $N_j(\varphi)$ converges in operator norm, then $N_j(\varphi)\rightarrow K_{1}$ on
$$A(\mathbb{D})=H(\mathbb{D})\cap \{f:\overline{\mathbb{D}}\rightarrow \mathbb{C}, \ continuous\},$$
where $K_{1}(f)=f(1)$ on $A(\mathbb{D})$.
The remaining of the proof is similar to that of \cite[Theorem 3.6]{beltran}, by considering $g(z)=\frac{1+z}{2}\in A(\mathbb{B}_n)$. 

\vspace*{0.5cm}

\textbf{Acknowledgments}. This paper was supported by the Iran National Science Foundation: INSF [project number 4000186].

\vspace*{0.5cm}

 Hamzeh Keshavarzi

E-mail: Hamzehkeshavarzi67@gmail.com

Department of Mathematics, College of Sciences,
Shiraz University, Shiraz, Iran.

\vspace*{0.5cm}

Karim Hedayatian

E-mail:
 hedayati@shirazu.ac.ir

Department of Mathematics, College of Sciences,
Shiraz University, Shiraz, Iran.

\end{document}